\documentclass[a4paper,12pt]{article}
\usepackage[utf8]{inputenc}
\usepackage{amsfonts}
\usepackage{amsmath}
\usepackage{mathrsfs}
\usepackage{amssymb}
\usepackage{amsthm}
\usepackage{indentfirst}
\usepackage{enumitem}
\usepackage{verbatim}
\usepackage{color}

\usepackage[a4paper,left=3cm,right=3cm,top=2.5cm,bottom=2.5cm]{geometry}

\title{Topological full groups of line-like minimal group actions are amenable}

\author{N\'ora Gabriella Sz\H{o}ke\thanks{Institut Fourier, Universit{\'e} Grenoble Alpes, France. Email: nora.gabriella.szoke@gmail.com. Research supported by the Swiss National Science Foundation, Early Postdoc.Mobility fellowship no.~P2ELP2\_184531.}}
\date{}

\theoremstyle{plain}

\newtheorem{theorem}{Theorem}[section]
\newtheorem{lemma}[theorem]{Lemma}
\newtheorem{claim}[theorem]{Claim}
\newtheorem{corollary}[theorem]{Corollary}
\newtheorem{proposition}[theorem]{Proposition}

\theoremstyle{definition}
\newtheorem{definition}[theorem]{Definition}

\theoremstyle{remark}

\newtheorem{remark}[theorem]{Remark}

\newcommand{\eps}{\varepsilon}
\newcommand{\dd}{\mathrm{d}}

\newcommand{\Z}{\mathbb{Z}}
\newcommand{\N}{\mathbb{N}}

\newcommand{\PP}{\mathscr{P}}

\newcommand{\acts}{\curvearrowright}
\newcommand{\hooklongrightarrow}{\mathrel{\lhook\mkern -3.5mu\relbar\mkern -4.5mu \rightarrow }}

\begin{document}

\maketitle

\begin{abstract}
We consider a finitely generated group acting minimally on a compact space by homeomorphsims, and assume that the Schreier graph of at least one orbit is quasi-isometric to a line. We show that the topological full group of such an action is amenable.
\end{abstract}

\section{Introduction}

\noindent Consider a group $G$ and a compact Hausdorff topological space $\Sigma$. A group action $G\acts \Sigma$ by homeomorphisms is called \emph{minimal} if $\Sigma$ has no proper $G$-invariant closed subset. The \emph{topological full group} $[[G\acts \Sigma]]$ is the group of all homeomorphisms of $\Sigma$ that are piecewise given by the action of elements of $G$, where each piece is open in $\Sigma$.

The notion of topological full groups was first introduced for $\Z$-actions by Giordano, Putnam and Skau \cite{GPS99}. Among others, Matui and Nekrashevych investigated these groups (\cite{Mat06}, \cite{Mat11}, \cite{Mat15}, \cite{Nek17}). Their results show that the derived subgroup of the topological full group is often simple, and in many cases it is also finitely generated.

In their groundbreaking paper \cite{JM13}, Juschenko and Monod developed a strategy for proving the amenability of topological full groups. They show that the topological full group of a minimal Cantor $\Z$-action is amenable. Combined with the results of Matui, their paper provides the first examples of finitely generated infinite simple amenable groups. A natural question arises: how far can we extend their technique? Several directions were investigated in \cite{JdlS15}, \cite{JNdlS16}, \cite{JMBMdlS18}, and by the author of the present paper in \cite{Szo21}.

The goal of this paper is to further stretch the Juschenko-Monod result in a certain direction. Namely, we consider a minimal action of a finitely generated group such that there exists an orbit that is quasi-isometric to a line, and show that the topological full group of such an action is amenable. This is a generalization of Theorem A in \cite{Szo21}, where the group was virtually cyclic.

\begin{theorem}\label{thm:main}
Let $G$ be a finitely generated group acting minimally on a compact Hausdorff topological space $\Sigma$ by homeomorphisms. Assume that there exists a $G$-orbit $X\subseteq \Sigma$, such that the Schreier graph of the action of $G$ on $X$ is quasi-isometric to $\Z$. Then the topological full group $[[G\acts \Sigma]]$ is amenable.
\end{theorem}

In order to illustrate the interest of our result, let us mention how to recover a result of Matte Bon about the Grigorchuk group. Let $G$ be the first Grigorchuk group (defined in \cite{Gri80}), which is usually defined as a transformation group of the binary rooted tree. Its action on the boundary of the tree - a Cantor set - is known to be minimal and its Schreier graphs are quasi-isometric to lines, as seen in \cite{GLN17}. Theorem \ref{thm:main} can be applied to deduce the following. 

\begin{corollary}[Matte Bon, \cite{MB15}]
The topological full group of the Grigorchuk group acting on the boundary of the rooted binary tree is amenable.
\end{corollary}

This result was first proved by Matte Bon, who showed that the Grigorchuk group can be embedded in the topological full group of a minimal Cantor $\Z$-action (\cite{MB15}).

 There are more groups to which our Theorem \ref{thm:main} can be applied, for instance the groups defined by Nekrashevych in \cite{Nek18}. Let $a$ be an involution on a Cantor space $\Sigma$. We say that a finite group $A$ of homeomorphisms of $\Sigma$ is a \emph{fragmentation} of $a$ if for all $h\in A$ and all $x\in \Sigma$, we have $h(x)=x$ or $h(x)=a(x)$ and for every $x\in \Sigma$ there exists $h\in A$ such that $h(x)=a(x)$. In \cite{Nek18} it is shown that for a fragmentation $A, B$ of a minimal action of the dihedral group $D_{\infty}=\langle a, b\rangle$, the action of the topological full group $G=\langle A, B\rangle$ is minimal. It is not difficult to see that the associated Schreier graphs are quasi-isometric to lines. Therefore, we get the following result of Nekrashevych as a corollary of Theorem \ref{thm:main}.
 
\begin{corollary}[Nekrashevych, \cite{Nek18}]
For any fragmentation $A, B$ of a minimal action of the dihedral group $D_{\infty}=\langle a, b\rangle$, the topological full group of $G=\langle A, B\rangle$ is amenable.
\end{corollary}

\vspace{10pt}

\noindent{\bf Acknowledgements.} It was Nicol\'as Matte Bon who asked me whether a similar result in my thesis could be true in this more general setting, I would like to thank him for this question. Furthermore, I am very grateful to Fran\c{c}ois Dahmani for our numerous discussions and his comments on a previous version of the paper.

\section{Preliminaries}

\subsection{Quasi-isometry}

If $\mathcal{G}=(\mathrm{V}(\mathcal{G}), \mathrm{E}(\mathcal{G}))$ is a connected graph, then we can think of $\mathcal{G}$ as a metric space. The distance $\mathrm{d}:\mathrm{V}(\mathcal{G})\times \mathrm{V}(\mathcal{G})\rightarrow \mathbb{N}$ is defined to be the length of the shortest path between two vertices.

Let $\mathcal{G}_1$, $\mathcal{G}_2$ be two connected graphs with distance functions $\mathrm{d}_1,\mathrm{d}_2$ respectively. Recall that the map $f\colon \mathrm{V}(\mathcal{G}_1)\rightarrow \mathrm{V}(\mathcal{G}_2)$ is a \emph{quasi-isometry} if there exist constants $\alpha\geq 1$, $\beta\geq 0$ and $\gamma\geq 0$ such that the following two properties hold.
\begin{enumerate}
\item For all $u,v\in \mathrm{V}(\mathcal{G}_1)$ we have
\[ \alpha^{-1}\  \mathrm{d}_1(u,v)-\beta\leq \mathrm{d}_2(f(u),f(v)) \leq \alpha\ \mathrm{d}_1(u,v)+\beta.\]
\item For every $w\in \mathrm{V}(\mathcal{G}_2)$ there exists $u\in \mathrm{V}(\mathcal{G}_1)$ such that $\mathrm{d}_2(f(u),w)\leq \gamma$.
\end{enumerate}
Two graphs are \emph{quasi-isometric} if there exists a quasi-isometry between them.

\subsection{Group actions and graphs}

As a convention, throughout the paper we always consider groups acting from the left.

Let $G$ be a group acting on a set $X$. The \emph{piecewise group} $\mathrm{PW}(G\acts X)$ of the action is defined as follows. A bijection $\varphi\colon X \to X$ is a \emph{piecewise $G$ map}, i.e., an element of the piecewise group, iff there exists a finite subset $S\subset G$ such that $\varphi(x)\in S\cdot x$ for every $x\in X$. In other words, we cut the space $X$ into finitely many pieces, and act on each of them with a group element. It is clear that the piecewise $G$ maps form a group.

If the group $G$ acts on a compact space by homeomorphisms, then the topological full group of this action is always a subgroup of its piecewise group. Indeed, by the compactness of the space a partition into open subsets is necessarily finite.

\vspace{3mm}

Let $G$ be a finitely generated group with a symmetric generating set $S$, and assume that $G$ acts on a set $X$. Recall that the \emph{Schreier graph} of this action $\mathrm{Sch}(G, X, S)$ is defined to be the graph with vertex set $X$ and edge set $\{ (x, s x) : x\in X, s\in S\}$. Sometimes it is also called the \emph{graph of the action} $G\acts X$.

Note that the Cayley graph of $G$ is the Schreier graph of its action on itself by (left) multiplication.

\begin{definition}
If $G$ is a finitely generated group with a fixed symmetric generating set $S$, then the \emph{length} of a group element $g\in G$ is defined as
\[ \mathrm{len}(g) = \min \{n\in \N : g = s_1s_2\dots s_n \text{ with } s_1, s_2, \dots, s_n\in S \}. \]
\end{definition}

In other words, the length of an element is its distance from the identity element in the Cayley graph.

\begin{definition}
For any graph $\mathcal{G}=(V, E)$ with distance function $\dd$ and a number $n\in\N$ we define the \emph{$n$-ball} around a point $p\in V$ to be
\[ B_n(p) = \{ q\in V : \dd(p,q)\leq n \}.\]
For a set $W\in V$, the \emph{$n$-neighborhood of $W$} is
\[ \Gamma_n(W) = \{ q\in V : \dd(q, W)\leq n\}.\]
If a group $G$ acts on the graph $\mathcal{G}$, then for a set of elements $D\subseteq G$ the \emph{$D$-neighborhood of a point} $p\in V$ is $D\cdot p = \{ d\cdot p : d\in D\}$, and the \emph{$D$-neighborhood of a set} $W\subseteq V$ is 
\[ D\cdot W = \bigcup_{q\in W} D\cdot q.\]
\end{definition}

Keep in mind that when $\mathcal{G}$ is a Schreier graph of a $G$-action, and $G=\langle S\rangle$, then the $n$-ball around a point is exactly the $S^n$-neighborhood of that point and the $n$-neighborhood of a set is equal to its $S^n$-neighborhood.

\subsection{Extensive amenability}

A group action $G\acts X$ is amenable if there exists a $G$-invariant mean on $X$. In the proof of our result we will use a stronger property, the extensive amenability of an action.

\begin{definition}
For a set $X$, let us denote the set of all finite subsets of $X$ by $\mathscr{P}_f(X)$. Note that this set becomes an abelian group with the symmetric difference. If a group $G$ acts on $X$, it gives rise to a $G$-action on $\mathscr{P}_f(X)$.

We say that the action $G\acts X$ is \emph{extensively amenable} if there exists a $G$-invariant mean on $\mathscr{P}_f(X)$ that gives full weight to the collection of sets containing any given finite subset of $X$.
\end{definition}

Extensively amenable actions were first used (without a name) in \cite{JM13}. The name was given in \cite{JMBMdlS18} and this concept turned out to be very useful for proving the amenability of topological full groups, see \cite{JM13}, \cite{JNdlS16}, \cite{JMBMdlS18}, or \cite{Szo21}. For a detailed introduction to extensive amenability, see Chapter 11 of \cite{BR18}. The following two statements about extensive amenability will be among the core ingredients in our proof.

\begin{proposition}[Proposition 3.6 in \cite{Szo21}]\label{prop:recurrent_orbits}
Let $G$ be a group acting on a set $X$. Assume that for all finitely generated subgroups $H\leq G$ and all $H$-orbits $Y\subseteq X$ the Schreier graph of the action $H\acts Y$ is recurrent. Then the action of the piecewise group $\mathrm{PW}(G\acts X)$ on $X$ is extensively amenable.
\end{proposition}

\begin{proposition}[Remark 1.5 in \cite{JMBMdlS18}]\label{prop:cocycle_amen_kernel}
Let $G\acts X$ be an extensively amenable action. Assume that there exists an embedding $G \hookrightarrow \mathscr{P}_f(X)$, $g\mapsto (c_g, g)$ such that the subgroup $\{ g\in G : c_g=\varnothing \}\leq G$ is amenable. Then $G$ itself is also amenable.
\end{proposition}

In Proposition \ref{prop:cocycle_amen_kernel}, such a map $c \colon G \to \mathscr{P}_f(X)$, $g\mapsto c_g$ is called a \emph{cocycle with amenable kernel}. Thus, we can rephrase the statement of the proposition as follows: If $G\acts X$ is an extensively amenable action, and there exists a cocycle on $G$ with amenable kernel, then $G$ is amenable. Let us denote the orbit of $p$ by $X$, i.e., $X=G\cdot p$.

\section{The proof}

In this section we consider a finitely generated group $G = \langle S \rangle$ acting minimally on a compact space $\Sigma$ satisfying the assumption in Theorem \ref{thm:main}. Let $X$ be an orbit such that the Schreier graph $\mathrm{Sch}(G, S, X)$ is quasi-isometric to $\Z$.

\subsection{Action on the orbit $X$}

The set $X$ is dense in $\Sigma$, since the action $G\acts \Sigma$ is minimal. Consider the restricted action of $G$ on $X$. We can define the embedding 
\begin{align*}
\eps_X\colon [[G\acts \Sigma]] &\hooklongrightarrow \mathrm{PW}(G\acts X);\\
\varphi & \longmapsto \varphi_{\big|X}.
\end{align*}
Since $X$ is dense in $\Sigma$, the $\varphi$-action on $X$ determines the $\varphi$-action on $\Sigma$, so the map $\eps_X$ is injective.

\begin{definition}\label{def:qi}
Let $\dd$ denote the distance function on the graph $X$. Let $f\colon X \to \Z$ be the quasi-isometry between $X$ and $\Z$. By definition, there exist constants $\alpha\geq 1$, $\beta\geq 0$ and $\gamma\geq 0$ such that
\begin{enumerate}
\item for all $x, y\in X$ we have
\[ \alpha^{-1} \cdot \dd(x,y) - \beta \leq | f(x)-f(y) | \leq \alpha\cdot  \dd(x,y) + \beta,\]
\item for every $n\in \Z$ there is $x\in X$ such that $| f(x)-n | \leq \gamma$.
\end{enumerate}
\end{definition}

\begin{lemma}\label{le:localfin}
For every $n\in \Z$, the set $f^{-1}(n)\subseteq X$ is finite.
\end{lemma}

\begin{proof}
If $f^{-1}(n)=\varnothing$, then it is finite. Now assume that it is non-empty. Let $x\in f^{-1}(n)$, then by Definition \ref{def:qi}, for any $y\in f^{-1}(n)$ we have 
\begin{align*}
 \alpha^{-1}\dd(x,y)-\beta  &\leq  |f(x)-f(y)| =0\\
 \dd(x,y) &\leq \alpha\beta.
\end{align*}
Hence, $f^{-1}(n)$ is contained in the ball of radius $\alpha\beta$ around $x$. This ball is a finite set since the graph $X$ is locally finite, so $f^{-1}(n)$ is also finite.
\end{proof}

The following two propositions are well-known for any graph that is quasi-isometric to $\Z$, but we include their proofs for completeness.

\begin{proposition}\label{prop:bi-inf}
There exists a bi-infinite geodesic in $X$.
\end{proposition}

\begin{lemma}\label{le:bi-inf}
Let $\mathcal{G}$ be a locally finite graph, i.e., the degree of every vertex is finite. The following are equivalent for a vertex $v\in \mathcal{G}$.
\begin{enumerate}[nolistsep]
\item For every $n\in \N$, there exists a geodesic of length $2n$ with midpoint $v$. 
\item There exists a bi-infinite geodesic through the vertex $v$.
\end{enumerate}
\end{lemma}

\begin{proof}
The implication $2\Rightarrow 1$~is clear. For the other direction, consider a vertex $v$ that satisfies the first statement.

Let us construct the rooted tree $\mathcal{T}$ as follows.  The vertices of $\mathcal{T}$ are the finite geodesics in $\mathcal{G}$ of even length with midpoint $v$. Two such geodesics are connected in $\mathcal{T}$ if their length difference is exactly 2 and the shorter one is a subset of the longer one. The root is the ``geodesic'' of length zero consisting only of the point $v$, and the $n$-th level of the tree consists of the geodesics of length $2n$. By the local finiteness of $\mathcal{G}$, the rooted tree $\mathcal{T}$ is also locally finite.

By K\H{o}nig's lemma, there exists an infinite ray in $\mathcal{T}$ from the root, say $\{v\}=\ell_0, \ell_1, \ell_2, \ell_3, \dots$, where the length of $\ell_i$ is $2i$ and $\ell_i\subseteq \ell_{i+1}$ for every $i\in \N$. Then the union $\bigcup_{i\in \N} \ell_i = \ell \subseteq \mathcal{G}$ is a bi-infinite geodesic in $\mathcal{G}$. This proves the implication $1\Rightarrow 2$.
\end{proof}

\begin{proof}[Proof of Proposition \ref{prop:bi-inf}.]
Let $f\colon X\to \Z$ be the quasi-isometry with constants $\alpha$, $\beta$, $\gamma$, and let us define $B= f^{-1}([0, \alpha+\beta])$. By Lemma \ref{le:localfin}, the set $f^{-1}(z)\subseteq X$ is finite for every $z\in \Z$. Consequently, $B$ is finite.

Let
\[ B_i = \{ x\in B\ | \ x \text{ is the midpoint of a length } 2i \text{ geodesic in } X\} \]
for all $i\in \N$. We have $B_{i+1}\subseteq B_i$ for every $i$.

We show that $B_i\neq \varnothing$ for every $i\in\N$. Consider a fixed $i\in \N$ and take $k\in\N$ such that $k\geq \alpha\cdot i + \beta$. Note that this choice ensures that $|f(x)-f(y)| > k$ implies $\dd(x,y) > i$ for some $x,y\in X$ by Definition \ref{def:qi}. Take $x_1, x_2\in X$ such that $f(x_1)< -k$ and $f(x_2)> \alpha+\beta+ k$. (It is possible to find such points by the second statement of Definition \ref{def:qi}.) Then $\dd(x_1, B)>i$ and $\dd(x_2, B)>i$ both hold.

Let $[x_1, x_2]$ be a shortest path between the two points, i.e., a geodesic. Note that if $u,v\in X$ are neighbors, then $|f(u)-f(v)|\leq \alpha+\beta$. Indeed, this holds by Definition \ref{def:qi}: $|f(u)-f(v)|\leq \alpha\ \dd(u,v)+\beta =\alpha+\beta$. Hence, by walking along the path $[x_1, x_2]$, the $f$-image changes by at most $\alpha+\beta$ at every step. On the other hand, we know that $f(x_1)<0$ and $f(x_2)>\alpha+\beta$, so there must be a point on this path $y\in [x_1, x_2]$, such that $0\leq f(y)\leq \alpha+\beta$. This implies that $y\in B$, so we have $\dd(x_1, y)>i$ and $\dd(y, x_2)>i$. Therefore, there exists a geodesic of length $2i$ with midpoint $y$, and hence $B_i\neq \varnothing$.

We proved that
\[B= B_0\supseteq B_1\supseteq B_2\supseteq \dots\]
is a decreasing sequence of non-empty finite sets. Therefore, their intersection is also non-empty. Take a point $x\in \bigcap_{i\in\N} B_i$. Then for every $i\in\N$, there exists a geodesic of length $2i$ in $X$ with midpoint $x$. Hence, by Lemma \ref{le:bi-inf}, there is a bi-infinite geodesic through the point $x$.
\end{proof}

\begin{proposition}\label{prop:m_geod}
There exists a constant $m\in\N$ such that $X$ is contained in the $m$-neighborhood of any bi-infinite geodesic in $X$.
\end{proposition}

\begin{proof}
Let $f\colon X \to \Z$ be the quasi-isometry with constants $\alpha$, $\beta$, $\gamma$ as in Definition \ref{def:qi}. Let $m=\alpha^2+2\alpha\beta$. Let $\ell$ be a bi-infinite geodesic in $X$, we would like to show that $X$ is contained in the $m$-neighborhood of $\ell$.

First, note that if $u,v\in X$ such that $\dd(u,v)>m= \alpha^2 +2\alpha\beta$, then we have 
\begin{align}
\alpha^{-1} \dd(u,v) -\beta & \leq |f(u)-f(v)| \nonumber \\
\alpha^{-1} m - \beta & < |f(u)-f(v)| \nonumber \\
\alpha +\beta & < |f(u)-f(v)| \label{eq:m_geod}
\end{align}

Take an arbitrary $x\in X$, and suppose for contradiction that $\dd(x, \ell) > m = \alpha^2 +2\alpha\beta$. This implies that for every $y\in \ell$, we have $\alpha +\beta < |f(x)-f(y)|$ by (\ref{eq:m_geod}).

Note that if $u,v\in X$ are neighbors, then $|f(u)-f(v)| \leq \alpha \ \dd(u,v) + \beta = \alpha +\beta$. Therefore, as we walk along the geodesic $\ell$, the $f$-image cannot jump over the value $f(x)$, since the distance of $f(\ell)$ from $f(x)$ is more than $\alpha+\beta$. Hence, $f(\ell)$ must be contained in a half-line, either $(-\infty, f(x)- \alpha -\beta)$ or $(f(x)+\alpha +\beta, +\infty )$.

We will show that $f(\ell)$ cannot be contained in a half-line. Without loss of generality, suppose that $f(\ell) \subseteq (N, +\infty )$, such that $N= \min f(\ell) \geq f(x)+\alpha +\beta$, and take $x_0\in \ell$ so that $f(x_0)=N$. Let $I= B_m(x_0) \cap \ell$ be a geodesic segment of length $2m$ on $\ell$ around $x_0$.

Take $y_1, y_2\in \ell\setminus I$ be in different components of $\ell \setminus I$ such that $f(y_1)\leq f(y_2)$. Consider $[x_0, y_2]$, which denotes a shortest path between the two points in $X$, in this case we may take the path that is contained in $\ell$. We know that $|f(u)-f(v)|\leq \alpha+\beta$ if $u$ and $v$ are neighbors, hence if we ``walk'' along the path $[x_0, y_2]$, the $f$-image changes by at most $\alpha+\beta$ in each step. Since $N=f(x_0) \leq f(y_1)\leq f(y_2)$, we can find $y_3\in [x_0, y_2]$, such that $|f(y_3)-f(y_1)|\leq \alpha + \beta$. On the other hand, $\dd(y_1, y_3)\geq \dd(y_1, x_0)> m = \alpha^2 +2\alpha\beta$, which is a contradiction by (\ref{eq:m_geod}).

Therefore, for every $x\in X$, we have $\dd(x, \ell) \leq m$, so $X$ is contained in the $m$-neighborhood of $\ell$.
\end{proof}

\subsection{Definition of the cocycle}

\begin{definition}\label{def:Y}
Let $f\colon X\to \Z$ be the quasi-isometry from Definition \ref{def:qi}. Let us define
\[Y = f^{-1}(\N) \subseteq X.\]
\end{definition}

For a subgraph $H$ of $X$ let us denote by $\partial H$ the vertices on the boundary of $H$, i.e., let
\[ \partial H = \{ x\in H : \text{there exists } y\in X\setminus H \text{ such that } (x,y)\in E(X) \} \subseteq H \subseteq X.\]

\begin{lemma}\label{le:boundY}
The set $Y$ is infinite and $\partial Y$ is finite.
\end{lemma}

\begin{proof}
By the second requirement in Definition \ref{def:qi}, we have that $f^{-1}(I)\neq \varnothing$ for every interval $I$ of length at least $2\gamma$. Since $\N$ contains infinitely many pairwise disjoint intervals of length $2\gamma$, the preimage $f^{-1}(\N)= Y$ is infinite.

For $x\in H$ and $y\in X\setminus H$ we have $(x,y)\in E(X)$ if and only if $\dd(x,y)=1$. By Definition \ref{def:qi}, we have
\[ \alpha^{-1}-\beta =\alpha^{-1}\dd(x,y) -\beta \leq |f(x)-f(y)|\leq \alpha\ \dd(x,y)+\beta=\alpha+\beta.\]
Therefore, since $f(x)\in \N$ and $f(y)\in \Z\setminus \N$, we must have $0\leq f(x)\leq \alpha+\beta-1$, so $\partial Y\subseteq f^{-1}([0,\alpha+\beta-1])$. The latter is a finite set by Lemma \ref{le:localfin}, and hence $\partial Y$ is also finite.
\end{proof}

\begin{lemma}\label{le:cocycle}
For every group element $g\in G$, the set $gY\setminus Y$ is finite.
\end{lemma}

\begin{proof}
Notice that for all $x\in X$ we have $\dd(x,gx)\leq \mathrm{len}(g)$. Therefore, the set $gY\setminus Y$ is contained in the $\mathrm{len}(g)$-neighborhood of $\partial Y$. Since $\partial Y$ is finite by Lemma \ref{le:boundY}, the $\mathrm{len}(g)$-neighborhood is also a finite set by the local finiteness of $X$. Hence, $gY\setminus Y$ is finite.
\end{proof}

\begin{proposition}\label{prop:cocycle}
For every piecewise map $\varphi\in \mathrm{PW}(G\acts X)$, the set $Y\triangle \varphi(Y)$ is finite.
\end{proposition}

\begin{proof}
There exists a finite set $T \subseteq G$ such that for every $x\in X$ we have $\varphi(x)\in T\cdot x$. Hence, we have the inclusion
\[ \varphi(Y)\setminus Y \subseteq \left( \bigcup_{t\in T} t Y\right) \setminus Y = \bigcup_{t\in T} (t Y\setminus Y).\]
By Lemma \ref{le:cocycle}, $tY\setminus Y$ is finite for all $t\in T$, so $\varphi(Y)\setminus Y$ is also finite. The same argument works for $\varphi^{-1}(Y)\setminus Y$, and hence $\varphi\left( \varphi^{-1}(Y)\setminus Y\right) = Y\setminus \varphi(Y)$ is finite as well. This implies that the set
\[ Y \triangle \varphi(Y) = (Y\setminus \varphi(Y))\cup (\varphi(Y)\setminus Y)\]
is also finite, finishing the proof.
\end{proof}

\begin{definition}\label{def:cocycle}
For $\varphi\in \mathrm{PW}(G\acts X)$ let us define
\[c_{\varphi}= Y\triangle \varphi(Y)\in \PP_f(X).\]
\end{definition}

\begin{remark}\label{rem:cocycle}
We defined the map $c\colon \mathrm{PW}(G\acts X) \to \mathscr{P}_f(X)$. This gives rise to the cocycle $c\colon [[G\acts \Sigma ]]\to \mathscr{P}_f(X)$. We would like to show that its kernel $\{\varphi\in [[G\acts \Sigma]] : c_{\varphi} = \varnothing \}$ is amenable in order to use this cocycle in Proposition \ref{prop:cocycle_amen_kernel}. Note that
\begin{align} 
\ker c &= \{\varphi\in [[G\acts \Sigma]] : c_{\varphi} = \varnothing \} \nonumber\\
&= \{\varphi\in [[G\acts \Sigma]] : Y\triangle \varphi(Y) = \varnothing \} \nonumber\\
&= \{\varphi\in [[G\acts \Sigma]] : \varphi(Y)=Y \} \nonumber\\
&= [[G\acts\Sigma]]_Y. \label{eq:kernel_stab}
\end{align}
Hence, the kernel of $c$ is exactly the stabilizer of the set $Y$ in the topological full group $[[G\acts \Sigma]]$. In the next sections we prove that this stabilizer is amenable.
\end{remark}

\subsection{Ubiquitous patterns in the action}

\begin{definition}
Let $G$ be a group acting on the space $\Sigma$. Let $D\subset G$ be a finite set containing the identity element. For an element $\varphi\in \mathrm{PW}(G\acts \Sigma)$ and for two points $q_1,q_2\in \Sigma$, we say that the \emph{$\varphi$-action is the same on the $D$-neighborhood of $q_1$ and $q_2$}, if the $D$-neighborhoods of $q_1$ and $q_2$ are isomorphic, and for every $d\in D$, $\varphi$ acts by the same element of $G$ on $d\cdot q_1$ and on $d\cdot q_2$, i.e., there exists $g\in G$ such that $\varphi(d\cdot q_1)=gd\cdot q_1$ and $\varphi(d\cdot q_2)=gd\cdot q_2$.
\end{definition}

\begin{lemma}\label{le:upp}
Let $G=\langle S\rangle$ be a group acting minimally on the compact space $\Sigma$ with a finite symmetric generating set $S$, and take an arbitrary point $q\in X$.

For every finite subset $F\subset [[G\acts \Sigma]]$ and every $n\in \N$, there exists $r=r(q, F, n)\in\N$ so that for every $y\in X$ there exists $z\in B_r(x)$ such that for all $\varphi\in F$, the $\varphi$-action is the same on the $S^n$-neighborhood of $q$ and $z$.
\end{lemma}

\begin{proof}
Let us fix the elements $\varphi_1,\dots, \varphi_k \in [[G\acts \sigma]]$ and a number $n\in \N$.

Choose a finite partition $\mathcal{P}$ of $\Sigma$ such that every $\varphi_i$ is acting with one element of $G$ when restricted to any element of $\mathcal{P}$. Then there exists an open neighborhood $V$ of $q$ such that the sets $g\cdot V$ for $g\in S^n$ are pairwise disjoint, and every $g \cdot V$ is contained in some element of $\mathcal{P}$. Since $V$ is open and non-empty, the union 
\[\bigcup_{g\in G} g\cdot V=\bigcup_{j\geq 1} \bigcup_{g\in S^{j}} g\cdot V\]
is non-empty, open and $G$-invariant, so by minimality we have
\[\bigcup_{j\geq 1} \bigcup_{g\in S^{j}} g\cdot V = \Sigma.\]
Due to the compactness of $\Sigma$, already a finite union must cover it, so there exists $j\in \N$ such that 
\[\bigcup_{g\in S^{j}} g\cdot V=\Sigma.\]
Let $r=j$. Now let $y\in X=G\cdot q$ be an arbitrary point. Then $y=h\cdot q$ for some $h\in G$. We have
\[ \Sigma=h^{-1}\cdot \Sigma= \bigcup_{g\in S^r}h^{-1}g\cdot V,\]
so there exists $\hat{g}\in S^r$ such that $q\in h^{-1}\hat{g}\cdot V$. This means that $\hat{g}^{-1}h\cdot q\in V$. Let $z=\hat{g}^{-1}h\cdot q$, and note that $z=\hat{g}^{-1}h\cdot q\in B_r (h\cdot q)=B_r(y)$. On the other hand, $q$ and $z$ are both in $V$, so for every $g\in S^n$, the points $g\cdot q$ and $g \cdot z$ are in the same element of the partition $\mathcal{P}$, so every $\varphi_i$ acts with the same element of $G$ on them. Therefore, for all $i=1,\dots,k$, the $\varphi_i$-action is the same on the $S^n$-neighborhood of $q$ and $z$.

This proves the statement of the lemma for $r$.
\end{proof}

\begin{lemma}\label{le:d_phi}
For every piecewise map $\varphi \in \mathrm{PW}(G \acts X)$, there exists a number $d_{\varphi} \in \N$, such that for every $x\in X$, $\dd(x, \varphi(x)) \leq d_{\varphi}$.
\end{lemma}

\begin{proof}
There exists a finite set $T\subseteq G$ such that for every $x\in X$, we have $\varphi(x)\in T\cdot x$. The statement of the lemma holds for $d_{\varphi} = max\{\mathrm{len}(t)\ : \ t\in T\}$.
\end{proof}

\begin{definition}\label{def:N_phi}
Let $m\in \N$ be the constant from Proposition \ref{prop:m_geod}. Let us fix a bi-infinite geodesic $\ell$ in $X$ (it exists by Lemma \ref{le:bi-inf}), and a point $p\in \ell$.

Let us denote the two ends of $\ell$ by $+\infty$ and $-\infty$. For a set $A\subseteq X$ we will say that $+\infty\in A$, if there exists a point $x\in \ell$ such that $[x, +\infty] \subseteq A$, where $[x,+\infty]\subseteq \ell$ denotes the half-line from $x$ towards $+\infty$. Similarly, $-\infty\in A$ if there exists $x\in \ell$ such that $[x, -\infty]\subseteq A$.

Let $R\in\N$ be such that the $R$-ball around the point $p$ contains both $\partial Y$ and $\partial Y^c$ (such a radius exists since $\partial Y$ is finite by Lemma \ref{le:boundY}, and hence $\partial Y^c$ is also finite).

For a piecewise map $\varphi \in \mathrm{PW}(G \acts X)$ let us define the number 
\[N_{\varphi} = 6m + R + 2d_{\varphi}.\]
\end{definition}

\begin{lemma}\label{le:oneend}
If a set $A\subseteq X$ and its complement $A^c$ are both infinite, but its boundary $\partial A$ is finite, then it contains exactly one end of $\ell$.
\end{lemma}

\begin{proof}
It is enough to prove that if $A$ is infinite and $\partial A$ is finite, then it contains at least one end of $\ell$. Indeed, we can apply this statement to both $A$ and $A^c$ -- since $\partial A^c$ is also finite -- to prove the statement of the lemma.

Suppose for contradiction that $A\subseteq X$ is an infinite set with finite boundary such that $+\infty, -\infty \notin A$. Since its boundary is finite, there exists a ball $B_r(x)$ with finite radius such that $\partial A\subseteq B_r(x)$. By the definition of the boundary, a connected component of the set $\ell\setminus B_r(x)$ must entirely belong either to $A$ or to $A^c$. Since $B_r(x)$ is finite, there exists a connected component of $\ell\setminus B_r(x)$ containing $+\infty$, and there is one (possibly the same) containing $-\infty$. Since we assumed that $+\infty, -\infty\notin A$, we have $+\infty, -\infty \in A^c$.

Now consider an arbitrary point $y\in X\setminus B_{r+m}(x)$, where $m$ is the constant from Proposition \ref{prop:m_geod}. By Proposition \ref{prop:m_geod}, there exists a point $\hat{y}\in \ell$ such that $\dd(y, \hat{y})\leq m$. Since $\hat{y}$ is connected to either $+\infty$ or $-\infty$ outside of $B_r(x)$ (and $\partial A\subseteq B_r(x)$), we must have $\hat{y}\in A^c$. We have $y\in X\setminus B_r(x)$, but we can say even more: there is a path of length at most $m$ connecting $y$ and $\hat{y}$ that lies outside of the ball $B_r(x)$. Since $y$ is connected to $\hat{y}$ outside of $B_r(x)$, it must also belong to $A^c$. 

Therefore, $X\setminus B_{r+m}(x)\subseteq A^c$, and hence $A\subseteq B_{r+m}(x)$. This contradicts the assumption that $A$ is infinite, so we must have $+\infty\in A$ or $-\infty\in A$.
\end{proof}

\begin{corollary}\label{cor:oneend}
The set $Y$ contains exactly one end of $\ell$, we can assume that $+\infty \in Y$, but $-\infty\notin Y$.
\end{corollary}

\begin{lemma}\label{le:stab}
Let us fix a finite subset $F \subseteq [[G \acts \Sigma]]_Y$ of the stabilizer of $Y$ and a number $n>\max\{N_{\varphi} : \varphi\in F\}$. Assume that there is a point $z\in X$ such that the $\varphi$-action is the same on the $S^n$-neighborhood of $p$ and of $z$ for every $\varphi \in F$. Then there exists a set $Y_z\subseteq  X$, such that $\partial Y_z\subseteq B_R(z)$, $+\infty\in Y_z$, $-\infty \in Y_z^c$ and $F \subseteq [[G\acts \Sigma]]_{Y_z}$.
\end{lemma}

\begin{proof}
Since the $\varphi$-action is the same on the $S^n$-neighborhood of $p$ and $z$ for every $\varphi\in F$, there exists a bijection 
\[h\colon S^n p \longrightarrow  S^n z,\]
such that for every $x\in S^n p=B_n(p)$, $\varphi$ acts by the same group element on the points $x$ and $h(x)$. Let us define $B^+ = h(S^np \cap Y)$ and $B^- = h(S^np \cap Y^c)$, and let
\begin{align*}
A^+ &= \{ x\in X : \text{there exists a path from } x \text{ to } B^+ \text{ that does not intersect } B^-\},\\
A^- &= \{ x\in X : \text{there exists a path from } x \text{ to } B^- \text{ that does not intersect } B^+\}.
\end{align*}
We will show that setting $Y_z = A^+$ or $Y_z=A^-$ satisfies the statement of the lemma.

\begin{claim} 
We have $A^- = (A^+)^c$. 
\end{claim}

\begin{proof}
Since $X$ is connected, every $x\in X$ is connected to some point of $B^+\cup B^- = S^nz=B_n(z)$, and hence $A^+ \cup A^- = X$. Therefore, we have to prove $A^+ \cap A^- = \varnothing$.

Suppose for contradiction that there exists a point that can be connected to both $B^+$ and $B^-$ without intersecting the other. This means that we can find points $z_+\in B^+$ and $z_-\in B^-$ that are connected by a path outside of $B_n(z)$, i.e., there exists a path $z_+=x_0, x_1, x_2, \dots, x_{k-1}, x_k=z_-$, such that $x_i \in X\setminus B_n(z)$ for $i= 1, \dots, k-1$.

For a point $x\in X$, we will denote its `projection' to $\ell$ by $\hat{x}$, i.e., the closest point to $x$ on $\ell$. If there are several such points, let us choose the closest one to the end $-\infty$. By Proposition \ref{prop:m_geod}, for every $x\in X$, we have $\dd(x, \ell)=\dd(x, \hat{x}) \leq m$.

We know that $\dd(z, x_i) \geq n >N_{\varphi}$ for $i=0,1,\dots, k$. By the triangle inequality, we have
\begin{align}
4m+R+2d_{\varphi} = N_{\varphi}-2m < n- 2m & \leq \dd(\hat{z}, \hat{x}_i) \leq n+2m. \label{eq:stablemma_triangle}
\end{align}
The two projections $\hat{z}_+$ and $\hat{z}_-$ are either separated by $\hat{z}$ on $\ell$ or they are on the same side of it. In both cases, we get a contradiction:

\begin{enumerate}[label={\arabic*.}]
\item Suppose that $\hat{z}_+$ and $\hat{z}_-$ are separated by $\hat{z}$ on $\ell$. Consider the points $x_i$ of the path connecting $z_+$ and $z_-$, and their projections $\hat{x}_i$. Since $\hat{z}_+=\hat{x}_0$ and $\hat{z}_-=\hat{x}_k$ are separated by $\hat{z}$, there exists $i$ such that $\hat{x}_i$ and $\hat{x}_{i+1}$ are also separated by $\hat{z}$ on $\ell$. For this $i$, we must have
\[ 2(4m+R+2d_{\varphi})\leq \dd(\hat{x}_i, \hat{x}_{i+1})\]
by (\ref{eq:stablemma_triangle}). On the other hand, $\dd(\hat{x}_i, x_i)\leq m$, and $\dd(x_{i+1},\hat{x}_{i+1}) \leq m$, and hence
\[ \dd(\hat{x}_i, \hat{x}_{i+1})\leq 2m+1,\]
this gives a contradiction.

\item Suppose that $\hat{z}_+$ and $\hat{z}_-$ are on the same side of $\hat{z}$ on $\ell$. Our goal is to find a path of length at most $12m$ connecting $z_+$ and $z_-$ that lies in the ball $B_n(z)$. 

Since $\dd(z_+, z)=n$, we can choose $y_+$ such that $\dd(z_+, y_+) = 3m$ and $\dd(y_+, z)= n-3m$. Let $[z_+, y_+]$ denote a shortest path between these two points in the graph $X$. Clearly this path lies in $B_n(z)$. We define $y_-$ similarly for $z_-$. Consider the projection $\hat{y}_+$. First, note that $\dd(y_+, \hat{y}_+)\leq m$ and $\dd(y_+, z)=n-3m$, and hence $[y_+, \hat{y}_+]\subseteq B_n(z)$. By the triangle inequality, we have $\dd(\hat{z}_+, \hat{y}_+)\leq 5m$, so $\hat{y}_+$ cannot be separated from $\hat{z}_+$ by $\hat{z}$ on $\ell$. Similarly, the point $\hat{y}_-$ is also on the same side of $\hat{z}$ and $[y_-, \hat{y}_-]\subseteq B_n(z)$.

Again by the triangle inequality, we have that
\begin{align*}
n-5m & \leq \dd(\hat{y}_+, \hat{z}) \leq n-m,\\
n-5m & \leq \dd(\hat{y}_-, \hat{z}) \leq n-m.
\end{align*}
Since $\hat{y}_+$ and $\hat{y}_-$ are not separated by $\hat{z}$, we must have $\dd(\hat{y}_+, \hat{y}_-) \leq 4m$, and a shortest path connecting them lies on the geodesic $\ell$, so it is contained in $B_n(z)$.

Now look at the path 
\[ P= [z_+, y_+]\cup [y_+, \hat{y}_+] \cup [\hat{y}_+, \hat{y}_-] \cup [\hat{y}_-, y_-]\cup [y_-, z_-].\]
We have seen that all sections of this path are contained in $B_n(z)$. Its length is at most $3m + m+ 4m +m +3m=12m$.

Therefore, there exists a path $P$ of length at most $12m$ connecting $z_+$ with $z_-$ that lies in $B_n(z)$. Since $\dd(z_+, z) = n > N_{\varphi} = 6m + R +2d_{\varphi}$, and also $\dd(z_-, z) > 6m +R + 2d_{\varphi}$, we have that $\dd(P, z) > R+ 2d_{\varphi}$. Consequently, taking the $h$-preimage of the path $P$, we have that $\dd(h^{-1}(P), p) > R + 2d_{\varphi}$. Since $\partial Y \subseteq B_R(p)$ (by Definition \ref{def:N_phi}), the path $h^{-1}(P)$ cannot intersect the boundary of $Y$, so it must lie entirely in $Y$ or in $Y^c$. This contradicts the assumption that $z_+\in B^+ = h(Y\cap B_n(p))$ and $z_-\in B^- = h(Y^c\cap B_n(p))$.
\end{enumerate}
This concludes the proof of the fact that $A^+ \cap A^- = \varnothing$.
\end{proof}

\begin{claim}\label{cl:stablemma_boundary}
We have $\partial A^+ = h(\partial Y)$ (and similarly $\partial A^- = h(\partial Y^c)$). 
\end{claim}

\begin{proof}
Since $\partial Y\subseteq B_R(p)$ and $R<n$, we have $h(\partial Y) \subseteq \partial A^+$. For the other direction, consider a point $x\in \partial A^+$. Then $x$ has a neighbor $y\in A^-$. There are four possibilities.
\begin{enumerate}
\item If $x,y\notin B_n(z)$, then there is a path from $x$ (going through $y$) to $B^-$ without touching $B^+$, so $x\in A^-$, this contradicts the fact that $A^+\cap A^- = \varnothing$.
\item If $x\notin B_n(z)$, $y\in B_n(z)$, then $x, y$ is a path from $x$ to $B^-$ that does not intersect $B^+$, and hence $x\in A^-$. This is again a contradiction.
\item If $x\in B_n(z)$, $y\notin B_n(z)$, then $y, x$ is a path from $y$ to $B^+$ without going through $B^-$, so $y\in A^+$, which is also a contradiction.
\item The only remaining possibility is $x,y\in B_n(z)$. In this case we have $x\in B^+$, $y\in B^-$, so $h^{-1}(x)\in Y$, $h^{-1}(y)\in Y^c$, and hence $x\in h(\partial Y)$.
\end{enumerate}

We have proved the equality $\partial A^+ = h(\partial Y)$. Similarly, one can prove that $\partial A^- = h(\partial Y^c)$.
\end{proof}

\begin{claim}\label{cl:stablemma_inv}
The sets $A^+$ and $A^-$ are both invariant under the action of $F$.
\end{claim}

\begin{proof}
It is enough to prove the $F$-invariance of $A^+$, the other statement follows from this. 

Take a point $x\in A^+$ and let us fix $\varphi\in F$. We distinguish three cases.
\begin{enumerate}
\item If $x\notin B_n(z)$, then $\dd(x, A^-)> d_{\varphi}$, since $\partial A^+\subseteq B_R(z)$ and $n>R+d_{\varphi}$. We know that the distance of $x$ and $\varphi(x)$ is at most $d_{\varphi}$, so we must have $\varphi(x)\in A^+$. Similarly, we have $\varphi^{-1}(x)\in A^+$.
\item If $x\in B_n(z)$, but $\varphi(x)\notin B_n(z)$, then we must have $\dd(x, X\setminus B_n(z)) \leq d_{\varphi}$. Hence, $\dd(x, B_R(z))> d_{\varphi}$ (since $n> R+ 2d_{\varphi}$), so we get $\dd(x, A^-)> d_{\varphi}$ again. Therefore, we have $\varphi(x)\in A^+$ and $\varphi^{-1}(x)\in A^+$.
\item If $x, \varphi(x)\in B_n(z)$, then we have  $\varphi(h^{-1}(x)) = h^{-1}(\varphi(x))$ since the $\varphi$-action is the same on the $S^n$-neighborhood of $p$ and of $z$. We know that $h^{-1}(x)\in Y$, so $\varphi(h^{-1}(x))\in Y$ by the $\varphi$-invariance of $Y$. Hence, we have $\varphi(x)\in h(Y\cap B_n(p)) \subseteq A^+$. Similarly, $\varphi^{-1}(x)\in A^+$.
\end{enumerate}
Therefore, we have $\varphi \in [[G\acts \Sigma]]_{A^+}$ and hence also $\varphi \in [[G\acts \Sigma]]_{A^-}$.
\end{proof}

\begin{claim}
$A^+$ and $A^-$ are both infinite. 
\end{claim}

\begin{proof}
Suppose for contradiction that $A^+$ is finite. This implies that all infinite components of $\ell\setminus B_R(z)$ belong to $A^-$ (the number of such components is one or two). Hence, we have that $\ell\cap (B_{n+m}(z)\setminus B_{n-2m}(z))\subseteq A^-$. Now consider any point $x\in B_n(z)\setminus B_{n-m}(z)$. There exists a projection $\hat{x}$, such that $\dd(x, \hat{x})\leq m$. Therefore, by the triangle inequality, we have 
\[\hat{x}\in \ell\cap (B_{n+m}(z)\setminus B_{n-2m}(z))\subseteq A^-.\]
Furthermore, the shortest path connecting $x$ to $\hat{x}$ lies outside of $B_R(z)$, so it cannot intersect $\partial A^+$, and hence $x\in A^-$.

We proved that $B_n(z)\setminus B_{n-m}(z)\subseteq A^-$. Therefore, $B_n(z)\setminus B_{n-m}(z)\subseteq B^-$, so $B_n(p)\setminus B_{n-m}(p)\subseteq Y^c$. Since $\partial Y^c\subseteq B_R(p)$, this implies that $X\setminus B_n(p)\subseteq Y^c$, so $Y$ is finite. This is a contradiction, hence $A^+$ is infinite. We can prove the same way that $A^-$ is also infinite.
\end{proof}

We showed that $A^-=(A^+)^c$, and that $A^+$ and $A^-$ are both infinite. Therefore, they both contain exactly one end of the geodesic $\ell$ by Lemma \ref{le:oneend}. Let us define
\[ Y_z = \begin{cases} A^+ & \text{ if } +\infty \in A^+,\\ 
A^- & \text{ if } +\infty \in A^-. \end{cases}\]
In both cases, we have $+\infty\in Y_z$, $-\infty \in Y_z^c$, $\partial Y_z\subseteq B_R(z)$ (by Claim \ref{cl:stablemma_boundary}) and $F\subseteq [[G\acts \Sigma]]_{Y_z}$ (by Claim \ref{cl:stablemma_inv}). This concludes the proof of the lemma.
\end{proof}

\subsection{Amenable kernel}

\begin{proposition}\label{prop:stab_loc_fin}
The stabilizer $[[G\acts \Sigma]]_Y$ is locally finite.
\end{proposition}

\begin{proof}
Consider a finite set $F\subseteq [[G \acts \Sigma ]]_Y$, our goal is to prove that the subgroup $\langle F\rangle$ is also finite. Define $N_F= \max\{ N_{\varphi} : \varphi\in F\}$. 

Let $n> N_F$. Let $r=r(p, F, n)$ from Lemma \ref{le:upp} for the point $p$, the finite set $F$ and the number $n$. Let $y_0 =z_0 = p$, and pick $y_i\in \ell$ for all $i\in \Z\setminus \{0\}$ such that $\dd(y_i, y_{i+1}) = 2r +2n + 2m + 2$ and $y_i$ is closer to $-\infty$ than $y_{i+1}$ for every $i\in \Z$. Now for every $i\in \Z\setminus \{0\}$ let us use Lemma \ref{le:upp} for the point $y_i$. Thus, we get the points $z_i\in B_r(y_i)$ (for all $i\in \Z$) such that for every $\varphi\in F$, the $\varphi$-action is the same on the $S^n$-neighborhood of $p$ and $z_i$. Note that due to the choice of the $y_i$'s, the $n$-balls around the points $z_i$ are pairwise disjoint.

Let $Y_0 = Y$ and for every $i\in \Z\setminus\{0\}$ let us use Lemma \ref{le:stab} for $F$, the point $z_i$ and the number $n$. For every $i$, there exists an infinite set $Y_i = Y_{z_i}\subseteq X$, such that we have $+\infty \in Y_i$, $-\infty\in Y_i^c$, furthermore $\partial Y_i\subseteq B_R(z_i)$ and $F\subseteq [[ G\acts \Sigma]]_{Y_i}$. 

\begin{claim}
For every $i\in \Z$, the set $Y_i$ contains $Y_{i+1}$. 
\end{claim}

\begin{proof}
Let us denote by $v$ the midpoint between $y_i$ and $y_{i+1}$ on $\ell$. Note that we have $\dd(v, B_n(z_i)) \geq m+1$ and $\dd(v, B_n(z_{i+1}) )\geq m+1$ by the choice of the distance between $y_i$ and $y_{i+1}$.

Therefore, $v\in Y_i$ since $+\infty\in Y_i$ and the half line $[v, +\infty]$ does not intersect $\partial Y_i \subseteq B_n(z_i)$. Similarly, we have $v\in Y_{i+1}^c$ since $-\infty\in Y_{i+1}^c$ and $[v, -\infty]$ does not intersect $\partial Y_{i+1}^c\subseteq B_n(z_{i+1})$.

Suppose for contradiction that there exists a point $x\in Y_{i+1}\setminus Y_i$. Let $\hat{x}$ be the closest point to $x$ on $\ell$, and let $[x, \hat{x}]$ denote a shortest path between them. We have $\dd(\hat{x}, x)\leq m$ by Proposition \ref{prop:m_geod}. There are two possibilites.
\begin{enumerate}[nolistsep]
\item If $\hat{x}\in [v, +\infty]$, then the path $[x, \hat{x}]\cup [\hat{x}, v]$ does not intersect $B_n(z_i)$ since $\dd(\hat{x}, B_n(z_i))\geq m+1$. However, $x\notin Y_i$ and $v\in Y_i$, so any path between the two must intersect $\partial Y_i\subseteq B_n(z_i)$. Hence, we get a contradiction.
\item If $\hat{x}\in [-\infty, v]$, then we can use a similar argument: We have $x\in Y_{i+1}$ but $v\notin Y_{i+1}$, so any path between them must intersect $\partial Y_{i+1}\subseteq B_n(z_{i+1})$. However, the path $[x, \hat{x}]\cup [\hat{x}, v]$ does not intersect $B_n(z_{i+1})$, leading to a contradiction.
\end{enumerate}
We get a contradiction in both cases, so such a point $x$ cannot exist. This proves that $Y_{i+1}\subseteq Y_i$.
\end{proof}

\begin{claim}\label{cl:stabprop_uniformbound}
For every $i\in \Z$, the set $Y_i\setminus Y_{i+1}$ is a finite set. Moreover, there is a uniform bound on the cardinality of the sets $Y_i\setminus Y_{i+1}$.
\end{claim}

\begin{proof}
Consider the $m$-neighborhood of the segment $[y_{i-1}, y_{i+2}]\subset \ell$, denoted by $B_m([y_{i-1}, y_{i+2}])$. First of all, the size of these sets has a uniform bound, since the graph $X$ is regular, and hence the size of the $m$-neighborhood of a set of $3(2r+2n+2m+2)+1$ points is uniformly bounded.

We show that $Y_i\setminus Y_{i+1} \subseteq B_m([y_{i-1}, y_{i+2}])$. Take any point $x\in Y_i\setminus Y_{i+1}$, let $\hat{x}$ denote its projection to $\ell$. Suppose for contradiction that $\hat{x}\in [y_{i+2}, +\infty]$, then $[\hat{x}, +\infty]$ does not intersect $\partial Y_{i+1}$ since $\partial Y_{i+1}\subseteq B_n(z_{i+1})\subseteq B_{n+r}(y_{i+1})$. Hence, we have $\hat{x}\in Y_{i+1}$, and also $x\in Y_{i+1}$, since $[x, \hat{x}]$ cannot intersect $\partial Y_{i+1}$ either. This contradicts the face that $x\in Y_i\setminus Y_{i+1}$. Therefore, we must have $\hat{x}\in [-\infty, y_{i+2}]$. Similarly, one can prove that if $\hat{x}\in [-\infty, y_{i-1}]$, then $x\notin Y_i$, leading to a contradiction again. This proves that $\hat{x}\in [y_{i-1}, y_{i+2}]$. Recall that $\dd(x, \hat{x}) \leq m$, and hence $x\in B_m([y_{i-1}, y_{i+2}])$.

We proved that $Y_i\setminus Y_{i+1} \subseteq B_m([y_{i-1}, y_{i+2}])$, this shows that the set $Y_i\setminus Y_{i+1}$ is finite for every $i\in \Z$, and that there is a uniform bound on their cardinalities.
\end{proof}

Notice that due to the $F$-invariance of every $Y_i$, the sets $Y_i\setminus Y_{i+1}$ are also $F$-invariant for all $i\in \Z$. Their union is the whole graph $X$, therefore, we can embed $\langle F\rangle$ into the direct product of the finite symmetric groups on the sets $Y_i\setminus Y_{i+1}$. By Claim \ref{cl:stabprop_uniformbound}, the size of these finite symmetric groups is uniformly bounded, and hence their direct product is locally finite.

Since $\langle F\rangle$ can be embedded into a locally finite group, and is finitely generated, it must be finite. This concludes the proof of the proposition.
\end{proof}

Finally, we are ready to prove our main theorem.

\begin{proof}[Proof of Theorem \ref{thm:main}]
Consider a minimal action $G\acts \Sigma$ of the finitely generated group $G$, such that there exists an orbit $X$ which is quasi-isometric to $\Z$.

First, we show that the action $[[G\acts \Sigma]] \acts X$ is extensively amenable. The Schreier graph of the action $G\acts X$ is quasi-isometric to $\Z$, and hence it is recurrent. As a corollary of Rayleigh's monotonicity principle, all connected subgraphs of a recurrent graph are also recurrent (for a proof see \cite{LyPer16}, Chapter 2). Thus, by Proposition \ref{prop:recurrent_orbits}, the action $\mathrm{PW}(G\acts X)\acts X$ is extensively amenable. It follows easily from the definition of extensive amenability that the action of any subgroup of $\mathrm{PW}(G\acts X)$ on $X$ is also extensively amenable. Therefore, $[[G\acts \Sigma]] \acts X$ is extensively amenable, as desired.

Next, we apply Proposition \ref{prop:cocycle_amen_kernel} for $[[G\acts \Sigma]]$ and $X$, with the cocycle $c$ defined in Definition \ref{def:cocycle}. By (\ref{eq:kernel_stab}) in Remark \ref{rem:cocycle}, we have $\ker c = [[G\acts \Sigma]]_Y$. By Proposition \ref{prop:stab_loc_fin}, this stabilizer $[[G\acts \Sigma]]_Y$ is locally finite, and hence it is amenable. Therefore, the conditions of Proposition \ref{prop:cocycle_amen_kernel} are satisfied: the action $[[G\acts \Sigma]] \acts X$ is extensively amenable, and the kernel of the cocycle $c$ is amenable. This proves that the topological full group $[[G\acts \Sigma]]$ is also amenable.
\end{proof}

\bibliographystyle{plain}
\bibliography{biblio_tits2}

\begin{thebibliography}{10}

\bibitem{BR18}
V.~Berth{\'e} and M.~Rigo.
\newblock {\em Sequences, {G}roups, and {N}umber {T}heory}.
\newblock Trends in Mathematics. Springer International Publishing, 2018.

\bibitem{GPS99}
T.~Giordano, I.~F. Putnam, and C.~F. Skau.
\newblock Full groups of {C}antor minimal systems.
\newblock {\em Israel J. Math.}, 111:285--320, 1999.

\bibitem{Gri80}
R.~Grigorchuk.
\newblock On burnside's problem on periodic groups.
\newblock {\em Funktsional. Anal. i Prilozhen}, 14(1):53--54, 1980.
\newblock English translation: Functional Anal. Appl. 14:41--43, 1980.

\bibitem{GLN17}
R.~Grigorchuk, D.~Lenz, and T.~Nagnibeda.
\newblock {\em Schreier Graphs of Grigorchuk's Group and a Subshift Associated
  to a Nonprimitive Substitution}, page 250–299.
\newblock London Mathematical Society Lecture Note Series. Cambridge University
  Press, 2017.

\bibitem{JdlS15}
K.~Juschenko and M.~de~la Salle.
\newblock Invariant means for the wobbling group.
\newblock {\em Bull. Belg. Math. Soc. Simon Stevin}, 22:281--290, 2015.

\bibitem{JMBMdlS18}
K.~Juschenko, N.~Matte~Bon, N.~Monod, and M.~de~la Salle.
\newblock Extensive amenability and an application to interval exchanges.
\newblock {\em Ergodic Theory Dynam. Systems}, 38:195--219, 2018.

\bibitem{JM13}
K.~Juschenko and N.~Monod.
\newblock Cantor systems, piecewise translations and simple amenable groups.
\newblock {\em Ann. of Math.}, 178:775--787, 2013.

\bibitem{JNdlS16}
K.~Juschenko, V.~Nekrashevych, and M.~de~la Salle.
\newblock Extensions of amenable groups by recurrent groupoids.
\newblock {\em Invent. Math.}, 206:837--867, 2016.

\bibitem{LyPer16}
R.~Lyons and Y.~Peres.
\newblock {\em Probability on trees and networks}, volume~42 of {\em Cambridge
  Series in Statistical and Probabilistic Mathematics}.
\newblock Cambridge University Press, 2017.

\bibitem{MB15}
N.~Matte~Bon.
\newblock Topological full groups of minimal subshifts with subgroups of
  intermediate growth.
\newblock {\em J. Mod. Din.}, 9:67--80, 2015.

\bibitem{Mat06}
H.~Matui.
\newblock Some remarks on topological full groups of {C}antor minimal systems.
\newblock {\em Int. J. Math.}, 17:231--251, 2006.

\bibitem{Mat11}
H.~Matui.
\newblock Homology and topological full groups of {\'e}tale groupoids on
  totally disconnected spaces.
\newblock {\em Proc. Lond. Math. Soc.}, 104:27--56, 2011.

\bibitem{Mat15}
H.~Matui.
\newblock Topological full groups of one-sided shifts of finite type.
\newblock {\em J. Reine Angew. Math.}, 705:35--84, 2015.

\bibitem{Nek18}
V.~Nekrashevych.
\newblock Palindromic subshifts and simple periodic groups of intermediate
  growth.
\newblock {\em Annals of Mathematics}, 187(3):667--719, 2018.

\bibitem{Nek17}
V.~Nekrashevych.
\newblock Simple groups of dynamical origin.
\newblock {\em Ergodic Theory Dynam. Systems}, 39:707--732, 2019.

\bibitem{Szo21}
N.~G. Sz\H{o}ke.
\newblock A {T}its alternative for topological full groups.
\newblock {\em Ergodic Theory Dynam. Systems}, 41(2):622--640, 2021.

\end{thebibliography}

\end{document}